\documentclass[review]{elsarticle}

\usepackage{lineno,hyperref}

\usepackage{epstopdf}

\modulolinenumbers[5]

\journal{Journal of \LaTeX\ Templates}
\usepackage{textcomp}
\usepackage{graphicx}              
\usepackage{amsmath}               
\usepackage{amsfonts}              
\usepackage{amsthm}
\usepackage{epstopdf}
\usepackage{algorithmicx}
\usepackage{algorithm}
\usepackage{algpseudocode}
\usepackage{dsfont}
\usepackage{xcolor}
\usepackage{subfig}
\newtheorem{thm}{Theorem}
\newtheorem{defn}{Definition}[section]

\newtheorem{cor}{Corollary}

\newtheorem{ex}{Example}
\newdefinition{rmk}{Remark}
\newproof{pf}{Proof}
\newproof{pot}{Proof of Theorem \ref{thm2}}

\begin{document}

\begin{frontmatter}

\title{An efficient spectral method for solving third-kind Volterra integral equations with non-smooth
solutions}

\author[1]{Y. Talaei}
\author[2]{P. M. Lima}
\address[1]{Department of Applied Mathematics, Faculty of Mathematical Sciences, University of Tabriz, Tabriz, Iran.}           
 \address[2]{ Centro de Matema'tica Computacional e Estoca'stica, Instituto
Superior Te'cnico, Universidade de Lisboa, Av. Rovisco Pais,
1049-001 Lisboa, Portugal.}
\ead{y_talaei@tabrizu.ac.ir (Corresponding author), pedro.t.lima@tecnico.ulisboa.pt}

\begin{abstract}
This paper is concerned with the numerical solution of the third kind Volterra integral equations with non-smooth solutions based on the recursive approach of the spectral Tau method.  To this end, a new set of the fractional version of canonical basis polynomials (called FC-polynomials) is introduced.
The approximate polynomial solution (called Tau-solution) is expressed in terms of FC-polynomials.
The fractional structure of Tau-solution allows recovering the standard degree of accuracy of spectral methods even in the case of non-smooth solutions. The convergence analysis of the method is studied. The obtained numerical results show the accuracy and efficiency of the method compared to other existing methods.
\end{abstract}

\begin{keyword}
Fractional recursive Tau method, Third kind Volterra integral equation, Fractional canonical polynomials,  Convergence analysis, Non-smooth solutions.
\end{keyword}

\end{frontmatter}
\section{Introduction}
Integral equations appeared for the first time in a work  by V. Volterra in 1884,  where he studied the  solution of an electrostatic problem \cite{PL1},and later used by the same author  in the modeling of population growth \cite{PL2}. Many mathematical models that arise  in various problems of physics, biology, chemistry, engineering, etc. are based on the integral equations \cite{Wazwaz,Corduneanu,Rahman}. Some of the most well-known numerical techniques used to approximate solutions of integral equations are: multi-step methods \cite{Paternoster,Jackiewicz,Conte2,Jackiewicz2}, spectral methods \cite{Paternoster,JShen,Brunner2017}, product integration methods \cite{Linz,Diogo}, Adomian decomposition method, homotopy perturbation method, Picard method \cite{modified,Wazwaz,Cherruault}, etc.

The linear Volterra integral equations (VIEs) of the general form
\begin{equation}\label{kmm}
a(t)y(t)=g(t)+\int_{0}^{t}k(t,s)y(s)ds, \ \ t\in I=[0,1],
\end{equation}
where $a(t)=0$ at a finite number of points in $I$ are called the third kind VIEs. In 1896, Volterra studied the solvability of the form (\ref{kmm}) when $k(0,0)=0$ and $k(0,0)\neq0$ on $(0,1]$ (\cite{NotaIII}, Nota III). The corresponding integral operator in Eq. (\ref{kmm}) is non-compact on $C(I)$ if $k(0,0)\neq0$ and its spectrum will be uncountable \cite{Brunner2017}.  This is a fundamental  property for investigating the existence and uniqueness of solution for Eq. (\ref{kmm}). In 1911, the third kind VIEs with various types of kernel singularities were studied by Evans in \cite{Evans}.
The third kind VIEs with weakly singular kernels have been considered in \cite{Pereverzev,Allaei1,Brunner2017} and references therein.

The aim of this paper is to present a numerical method for solving a class of third kind VIEs with weakly singular kernels
\begin{equation}\label{1}
t^{\beta}y(t)=t^{\beta}g(t)+\int_{0}^{t}(t-s)^{\gamma -1}s^{\beta-\gamma}H(t,s)y(s)ds,\ \ \ \ \ \ t\in I,\ \
\end{equation}
with $\beta\geq\gamma,\ \ 0<\gamma \leq 1,\ \beta>0$ and assume that
\begin{equation}\label{cvcvx}
\gamma=\frac{p_{1}}{q_{1}} ,\ \ \ \ \beta=\frac{p_{2}}{q_{2}},\ \
\end{equation}
where $ \ p_{i},q_{i}\in \mathds{N}$ and $lcm(p_{i},q_{i})=1$ for $i=1,2$. We denote the least common multiple of two positive integers $a$ and $b$ by $lcm(a,b)$. The given functions $g$ and $H$ are continuous on the domains $I=[0,1]$  and $ D:=\lbrace (t,s): t\in I, 0\leq s \leq t\rbrace $, respectively, and $y(t)$ is the unknown function.  Eq. (\ref{1}) can be written in the form of the equivalent cordial integral equation
 \begin{equation}\label{100}
y(t)=g(t)+\mathcal{K}y(t), \ \
\end{equation}
where
\begin{equation}\label{101}
\mathcal{K}y(t)=\int_{0}^{t}t^{-\beta}(t-s)^{\gamma -1}s^{\beta-\gamma}H(t,s)y(s)ds.
\end{equation}
is called cordial Volterra integral operator. This is a bounded linear operator on $C(T)$. In 2015,
 Allaei et. al. \cite{Allaei1} studied the existence and uniqueness of the solutions of Eq. (\ref{100}):
\begin{thm}\label{ml} The cordial integral operator $\mathcal{K}$ with $\beta\geq \gamma$ is compact if $H(0,0)=0$,  otherwise, it is a non-compact operator
with the uncountable spectrum
\[
\bigtriangleup_{\mathcal{K}}=\lbrace 0\rbrace \cup \lbrace H(0,0)B(\gamma,1-\gamma+\beta+\lambda); \ Re (\lambda)\geq 0 \rbrace.
\]
Also, the Eq. (\ref{100}) has a unique solution $y\in C(I)$ if $1 \notin \bigtriangleup_{\mathcal{K}}$. Here, $B(.,.)$ denotes the Beta function and
\[
\mathcal{K}t^{\lambda}=\widehat{\Theta}(\lambda)t^{\lambda};\ \ \widehat{\Theta}(\lambda):=\int_{0}^{1}\Theta(t)t^{\lambda}dt,
\]
with $\Theta(t)=t^{\beta-\gamma}(1-t)^{\gamma-1}$.
\end{thm}
In the last years, some numerical methods have been introduced to solve the Eq. (\ref{1}) including: collocation method with modified graded mesh \cite{Allaei1}, operational matrix method with hat functions \cite{Nemati2}, multi-step collocation method \cite{Shayanfard}, collocation method \cite{Song},  Legendre-Galerkin method \cite{Cai}, Bernstein approximation technique \cite{Usta}. However, there is very little work on the third kind Volterra  equations with non-smooth solutions via spectral methods.

The spectral methods are a class of applicable numerical techniques for obtaining approximate solutions of functional equations, based on  polynomials basis functions such as Chelyshev, Legendre, Jacobi, atc.
The three most commonly used techniques among spectral methods are: Galerkin, collocation, and Tau method \cite{JShen}. These methods have exponential rate of convergence in solving the problems with smooth solutions \cite{Ca, Brunner2017,Hesthaven}.
 Since spectral methods  with usual polynomial basis have low convergence order in the case of non-smooth solutions, we need to use appropriate basis  in this case. A useful technique to solve this problem is to use a fractional version of polynomial basis functions, see  \cite{camjournal,Maleknejad,Legendre,Calcolo,Talaei1} and references therein. Based on this motivational background,  the main focus of this paper is to develop a new version of recursive approach to the Tau method to solve Eq. (\ref{1}) by introducing a fractional set of the canonical polynomial basis (FC-polynomials). The FC-polynomials are constructed by a simple recursive algorithm. The approximate solution of the problem is obtained as a linear combination of FC-polynomials that is called Tau-solution. The unknown coefficients ($\tau$ parameters) in the Tau-solution are calculated by solving a linear algebraic system.

This paper contains the following sections:
\begin{itemize}
\item Section \ref{Sec2}: Some definitions and Theorems about shifted fractional Legendre polynomials on $[0,1]$ are presented.
\item Section \ref{Sec3}: A new fractional version of recursive Tau method to solve Eq. (\ref{1}) is introduced.
\item Section \ref{Sec4}:  The convergence of this method is analyzed.
\item Section \ref{Sec5}: Some examples are given to show the accuracy of the method in comparison with other existing methods.
\item Section \ref{Sec6}:   Conclusions and further work.
\end{itemize}
\section{Shifted fractional order Legendre polynomials}\label{Sec2}
The shifted fractional Legendre polynomials (FLPs) $\lbrace P_{i,\theta}(t)\rbrace_{i=0}^{\infty}$ on $[0,1]$ are the eigenfunctions of a fractional singular Sturm-Liouville equation
\begin{equation}
\frac{1}{\omega_{\theta}(t)}\mathbb{\partial}_{t}\lbrace \theta^{-1}(1-t^{\theta})t \mathbb{\partial}_{t}u(t)\rbrace=-i(i+1)u(t),
\end{equation}
where $\omega_{\theta}(t)=\theta t^{\theta-1}$ and $0<\theta\le 1$. The FLPs are orthogonal with respect to $\omega_{\theta}(t)$, namely,
\begin{equation}\label{oppo}
\int_{0}^{1}P_{i,\theta}(t) P_{j,\theta}(t)\omega_{\theta}(t)dt=\frac{\delta_{ij}}{2i+1}.
\end{equation}
and can be generated by the three-term recurrence relation as follows
\begin{align*}
&P_{i+1,\theta}(t)=\frac{(2i+1)(2t^{\theta}-1)}{i+1}P_{i,\theta}(t)-\frac{i}{i+1}P_{i-1,\theta}(t),\ \ \ i=1,2,...,\\
&P_{0,\theta}(t)=1,\ \ \ P_{1,\theta}(t)=2t^{\theta}-1.
\end{align*}
The explicit form of the FLPs is given by
\begin{equation}
P_{i,\theta}(t)=\sum_{j=0}^{i}C_{i,j}t^{j\theta};\ \ \ C_{i,j}=\frac{(-1)^{i+j}(i+j)!}{(i-j)!(j!)^{2}},\ \ \ \ i=0,1,...\ .
\end{equation}
Define
\begin{itemize}
\item $L^{2}[0,1]:=\lbrace u\vert u:[0,1]\rightarrow\mathds{R}_{+};\ \bigg(\displaystyle\int_{[0,1]}\vert u(t)\vert^2dt\bigg)^{1/2}<\infty \rbrace,$
\item $\mathbb{M}^{\theta}_{n}:=span\lbrace 1,t^{\theta},...,t^{n\theta}\rbrace$,
\item $\mathbb{W}^{\theta}_{n}:=span\lbrace P_{0,\theta}(t),...,P_{n,\theta}(t)\rbrace$,
\item $\widehat{\mathbb{W}}^{\theta}_{n}:=\mathbb{W}^{\theta}_{n}\otimes\mathbb{W}^{\theta}_{n}$.
\end{itemize}
From orthogonality condition (\ref{oppo}), the best approximation for $u(t)\in L^{2}[0,1]$ in the space $\mathbb{W}^{\theta}_{n}$ can be derived in the following form
\begin{equation}\label{tyyo}
u_{n}(t)=\sum_{i=0}^{n}a_{i}P_{i,\theta}(t)\simeq u(t),
\end{equation}
with
\[
a_{i}=(2i+1)\int_{0}^{1}u(t)P_{i,\theta}(t)\omega_{\theta}(t)dt,\ \ \ i=0,1,...,n.
\]
Similarly, for any two-variable function $U(t,s)\in L^{2}([0,1]\times [0,1])$, the best approximation in the space $\widehat{\mathbb{W}}^{\theta}_{n}$ can be derived as follows
\begin{equation}\label{TWo}
U_{n}(t,s)=\sum_{i=0}^{n}\sum_{j=0}^{n}a_{ij}P_{i,\theta}(t)P_{j,\theta}(s)\simeq U(t,s).
\end{equation}
in which
\[
a_{ij}=(2i+1)(2j+1)\int_{0}^{1}\int_{0}^{1}u(t,s)P_{i,\theta}(t)P_{j,\theta}(s)\omega_{\theta}(t)\omega_{\theta}(s)dtds,
\]
for $i,j=1,...,n$. See further details in \cite{JShen,Shenm,Legendre}.
\begin{thm}(One-variable fractional Taylor series \cite{Odibat})\label{MOpp}
 Let $0<\theta< 1$ and
\[
\mathbb{\partial}_{t}^{i\theta}u(t)\in C(0,1], \ i=0,...,n+1.
\]
Then,
\begin{equation}\label{1231}
u(t)=\sum_{i=0}^{n}\frac{\mathbb{\partial}_{t}^{i\theta}u(0)}{\Gamma(i\theta+1)}t^{i\theta}+\frac{t^{(n+1)\theta}}{\Gamma((n+1)\theta+1)}\mathbb{\partial}_{t}^{(n+1)\theta}u(t)\vert_{t=\xi},
\end{equation}
where $0 < \xi \leq t$, $\forall t\in (0,1]$. Here, the operator $\mathbb{\partial}_{t}^{\theta}$ denotes the  Caputo fractional derivative \cite{Kai}
\begin{equation}
\mathbb{\partial}_{t}^{\theta}u(t):=\frac{1}{\Gamma (\lceil \theta \rceil - \theta)}\displaystyle\int_{0}^{t}(t-s)^{\lceil \theta \rceil - \theta-1}u^{\lceil \theta \rceil}(s)ds.
\end{equation}
\end{thm}
\begin{thm}\cite{Maleknejad}(Two-variable fractional Taylor series)
Let $0<\theta\le 1$ and
\[
 \mathbb{\partial}_{t}^{i\theta}\mathbb{\partial}_{s}^{j\theta}U(t,s)\in C(0,1], \ i,j=0,...,n+1.
\]
 Then,
\begin{align}
U(t,s)&=\sum_{i=0}^{n}\sum_{j=0}^{i}\frac{\mathbb{\partial}_{t}^{(i-j)\theta}\mathbb{\partial}_{s}^{j\theta}U(0,0)}{\Gamma((i-j)\theta+1)\Gamma(j\theta+1)}t^{(i-j)\theta}s^{j\theta}\nonumber\\
&+\sum_{i=0}^{n+1}\frac{\mathbb{\partial}_{t}^{(n+1-i)\theta}\mathbb{\partial}_{s}^{i\theta}U(\xi,\eta)}{\Gamma((n+1-i)\theta+1)\Gamma(i\theta+1)}t^{(n+1-i)\theta}s^{i\theta}
\end{align}
where $0 < \xi \leq t$, $0 < \eta \leq s$, for all $t,s\in (0,1]$.
\end{thm}
\begin{thm}\label{098}
Suppose that $\mathbb{\partial}_{t}^{i\theta}u(t)\in C(0,1], \ i=0,...,n+1$,  $0<\theta< 1$ and $u_{n}$ be the best approximation of $u$ in the space $\mathbb{W}^{\theta}_{n}$ defined by (\ref{tyyo}), then the following error bound is valid
\begin{equation}\label{bon}
\Vert u-u_{n}\Vert_{\infty}  \leq \frac{M}{\Gamma((n+1)\theta+1)},
\end{equation}
where $M:=\underset{t\in (0,1]}{\max} \vert \mathbb{\partial}_{t}^{(n+1)\theta}u(t) \vert $  and $\Gamma(.)$ denotes the Gamma function.
\end{thm}
\begin{proof}
We have
\[
\sum_{i=0}^{n}\frac{\mathbb{\partial}_{t}^{i\theta}u(0)}{\Gamma(i\theta+1)}t^{i\theta} \in \mathbb{W}^{\theta}_{n},
\]
therefore, from Theorem \ref{MOpp} and assuming $u_{n}$ be the best approximation of $u$ we can write
\[
 \Vert u-u_{n} \Vert _{\infty} \leq \Vert u-\sum_{i=0}^{n}\frac{\mathbb{\partial}_{t}^{i\theta}u(0)}{\Gamma(i\theta+1)}t^{i\theta}\Vert_{\infty}\leq \frac{M}{\Gamma((n+1)\theta+1)}.
\]
\end{proof}
\begin{thm}\label{uuzz} \cite{Maleknejad}
Suppose that
\[
 \mathbb{\partial}_{t}^{i\theta}\mathbb{\partial}_{s}^{j\theta}U(t,s)\in C(0,1], \ i,j=0,...,n+1.
\]
and $U_{n}$ be the best approximation of $U$ in the space $\widehat{\mathbb{W}}^{\theta}_{n}$ defined by (\ref{TWo}), then the following error bound is presented
\begin{equation}\label{bon1}
\Vert U-U_{n} \Vert_{\infty}\leq \widehat{M}\sum_{i=0}^{n+1}\frac{1}{\Gamma((n+1-i)\theta+1)\Gamma(i\theta+1)}
\end{equation}
where
\[
\widehat{M}=\underset{0\leq i\leq n+1}{\max}\  \underset{t,s \in [0,1]}{\max} \vert \mathbb{\partial}_{t}^{(n+1-i)\theta}\mathbb{\partial}_{s}^{i\theta}U(t,s) \vert.
\]
\end{thm}
\section{Fractional recursive Tau method}\label{Sec3}
 The Tau method was introduced by Lanczos in 1938 \cite{Lanczos1}. The key idea of this method is to make a polynomial solution by adding a perturbation term  to the right-hand side of the problem.
The concept of the canonical polynomials was first introduced in this method as approximation solution basis.
 In 1969, Ortiz \cite{Ortiz} developed the recursive approach of the Tau method to a general class of ordinary differential equations.
  This approach to the Tau method was extended later to solve a certain class of functional problems see \cite{Khajah,ElDaou1} and references therein. In  \cite{camjournal, Calcolo} the authors studied and investigated the recursive approach to the Tau method to solve a class of weakly singular Volterra integral equations.
 In this section, we intend to carry out a new formulation of the Tau method to solve Eq. (\ref{1}) following the idea in \cite{Calcolo}.  Define the following notations:
\begin{align}
\left\{
  \begin{array}{ll}
 \delta:=lcm(q_{1},q_{2}),\ \ \\
 \theta_{1}:=\frac{\delta}{q_{1}},\ \ \ \theta_{2}:=\frac{\delta}{q_{2}},\\
 \sigma_{1}:=p_{1}\theta_{1},  \ \ \ \sigma_{2}:=p_{2}\theta_{2},\\
 \alpha:=\frac{1}{\delta},
  \end{array}
\right.
\end{align}
therefore, we can write
\[
\gamma=\sigma_{1}\alpha,\ \ \ \beta=\sigma_{2}\alpha.
\]
Assume that $\widetilde{g}(t)\in \mathbb{W}^{\alpha}_{n}$ and $\widetilde{H}(t,s)\in \widehat{\mathbb{W}}^{\alpha}_{n}$ are the best approximation of $H(t,s)$  and $g(t)$, respectively,
\begin{align}\label{90h}
&H(t,s)\simeq \widetilde{H}(t,s)=\sum_{i=0}^{n}\sum_{j=0}^{n}a_{ij}P_{i,\alpha}(t)P_{j,\alpha}(s)=\sum_{i,j=0}^{n}h_{i,j}t^{i \alpha}s^{j\alpha},\nonumber\\
& g(t)\simeq \widetilde{g}(t)=\sum_{i=0}^{n}g_{i}t^{i\alpha}.
\end{align}
Now, we define a linear operator $ \widetilde{\mathcal{L}} $ (called cordial operator) related to kernel function $\widetilde{H}(t,s)$ as follows
\begin{equation}\label{bpb}
\left\{
  \begin{array}{ll}
 \widetilde{\mathcal{L}}:\mathbb{M}^{\alpha}_{n}\rightarrow \mathbb{M}^{\theta}_{m},\ \ (n\leq m), \ \\
 (\widetilde{\mathcal{L}}y)(t):=y(t)-\displaystyle\int_{0}^{t}t^{-\sigma_{2}\alpha}(t-s)^{\sigma_{1}\alpha -1}s^{(\sigma_{2}-\sigma_{1})\alpha}\widetilde{H}(t,s)y(s)ds.
  \end{array}
\right.
\end{equation}
By applying $ \widetilde{\mathcal{L}} $ on $t^{r\alpha}$, we have
 \begin{align}\label{ZBN}
\widetilde{\mathcal{L}}(t^{r\alpha})&=t^{r\alpha}- \displaystyle   \int_{0}^{t}t^{-\sigma_{2}\alpha}(t-s)^{\sigma_{1}\alpha-1}s^{(\sigma_{2}-\sigma_{1})\alpha}\widetilde{H}(t,s)y(s)ds \nonumber\\
&=t^{r\alpha}-\displaystyle\sum_{i,j=0}^{n}h_{i,j}B\left(\sigma_{1}\alpha, (r+j+\sigma_{2}-\sigma_{1})\alpha +1\right)t^{(r+i+j)\alpha}
\nonumber\\
&:=\displaystyle\sum_{\ell=r}^{r+\vartheta_{n}}S_{\ell,r}t^{\ell\alpha},\ \ \ \  \vartheta_{n}\in \lbrace 0,1,..., 2n\rbrace.
\end{align}
 The main idea of the Tau method \cite{Ortiz} is to find a polynomial solution $y_{n}(t)\in \mathbb{M}^{\alpha}_{n}$ which is the exact solution (called Tau-solution) of the perturbed problem
\begin{align}\label{HOO}
\widetilde{\mathcal{L}}y_{n}(t)=\widetilde{g}(t)+\mathcal{H}_{n}(t).
\end{align}
The polynomial $ \mathcal{H}_{n}(t)$ is called a perturbation term. It follows from (\ref{ZBN}) that
\[
\deg \big[\widetilde{\mathcal{L}}y_{n}(t)-\widetilde{g}(t)\big]\leq (n+\vartheta_{n})\alpha,
\]
in which $\vartheta_{n}$ is called the height of the operator $\widetilde{\mathcal{L}}$. Thus, $ \mathcal{H}_{n}(t)$  can be defined in terms of FLPs as follows
\begin{equation}\label{bm1}
\mathcal{H}_{n}(t)=\sum_{r\in \mathcal{S}}^{}\tau_{n,r}P_{n+\vartheta_{n}-r,\alpha}(t).
\end{equation}
The set $\mathcal{S}$ and the unknown parameters $\tau_{n,r}$ are determined when finding the Tau-solution of Eq. (\ref{HOO}). Define the cordial Volterra integral operator with respect to $\widetilde{H}(t,s)$ as
 \begin{equation}
\widetilde{\mathcal{K}}y(t):=\int_{0}^{t}t^{-\beta}(t-s)^{\gamma -1}s^{\beta-\gamma}\widetilde{H}(t,s)y(s)ds.
\end{equation}
According to Theorem (\ref{ml}), the sufficient condition for the existence of unique solution to the perturbed problem (\ref{HOO}) is that
$1 \notin \bigtriangleup_{\widetilde{\mathcal{K}}}$, i.e.,
\begin{equation}\label{Uniq}
h_{0,0}B(\gamma,1-\gamma+\beta+r\alpha)\neq 1,\ \ \ \ \ \ r=0,1,...\ .
\end{equation}
\begin{defn}\label{CBN} For all $r\geq 0$, the $\varphi_{r}(t)$ are called fractional canonical polynomials (FC-polynomials) associated with a linear operator if $\widetilde{\mathcal{L}}$ if
\begin{equation*}
(\widetilde{\mathcal{L}}\varphi_{r})(t)=t^{r\alpha}.
\end{equation*}
\end{defn}
\begin{thm}
Assume that $\vartheta_{n}>0$ and all of the above notations and condition (\ref{Uniq}) hold. Then the FC-polynomials are generated by a recursive relation of the form
\begin{align}\label{Y2}
  \varphi_{r+\vartheta_{n}}(t)=\displaystyle\frac{1}{S_{r+\vartheta_{n},r}}\left(t^{r\alpha}-\displaystyle\sum_{\ell=r}^{r+\vartheta_{n}-1}S_{\ell,r}\varphi_{\ell}(t) \right),\ \ \ r=0,1,...,\ .
\end{align}
 In particular, for $\vartheta_{n}=0$ we have
\begin{equation}\label{vmmn}
 \varphi_{r}(t)=\frac{1}{S_{r,r}}t^{r\alpha}, \ \ \ r=0,1,...,
\end{equation}
in which
\[
 S_{r,r}=1-h_{0,0}B\left(\sigma_{1}\alpha, (r+\sigma_{2}-\sigma_{1})\alpha +1\right).
\]
\end{thm}
\begin{proof}
By using of Definition \ref{CBN} in  (\ref{ZBN}) we get
\begin{align*}\label{pop}
\widetilde{\mathcal{L}}(t^{r\alpha})&=\sum_{\ell=r}^{r+\vartheta_{n}}S_{\ell,r}t^{\ell\alpha}=S_{r+\vartheta_{n},r}t^{(r+\vartheta_{n})\alpha}+\sum_{\ell=r}^{r+\vartheta_{n}-1}S_{\ell,r}\widetilde{\mathcal{L}}\varphi_{r}(t).
 \end{align*}
From the linearity of operator $\widetilde{\mathcal{L}}$ we obtain
 \begin{align*}
\widetilde{\mathcal{L}}\left(t^{r\alpha}-\sum_{\ell=r}^{r+\vartheta_{n}-1}S_{\ell,r}\varphi_{\ell}(t)\right)=S_{r+\vartheta_{n},r}t^{(r+\vartheta_{n})\alpha},
 \end{align*}
so
\begin{equation*}
\varphi_{r+\vartheta_{n}}(t)=\frac{1}{S_{r+\vartheta_{n},r}}\left(t^{r\alpha}-\sum_{\ell=r}^{r+\vartheta_{n}-1}S_{\ell,r}\varphi_{\ell}(t) \right), \ \  r=0,1,...,
\end{equation*}
because of Definition \ref{CBN}.
In view of the above process, we can derive (\ref{vmmn}) for $\vartheta_{n}=0$.
\end{proof}
From relation (\ref{Y2}), the FC-polynomials $\lbrace \varphi_{r}(t)\rbrace_{r=\vartheta_{n}}^{\infty} $ are generated by the finite set of polynomials $\lbrace \varphi_{0}(t),...,\varphi_{\vartheta_{n}-1}(t)\rbrace$ are called undefined FC-polynomial. Therefore, one can rewrite them as follows
\begin{equation}\label{zcv}
\varphi_{r}(t)=\psi_{r}(t)+\sum_{j=0}^{\vartheta_{n}-1}d_{r,j}\varphi_{j}(t),
\end{equation}
where $\psi_{r}(t)$ are called associated FC-polynomials. Now, set
\begin{align}\label{Y1}
  \psi_{r}(t)=0,\ \ \ \displaystyle d_{r,j}=\delta_{r,j};\ \ \ r,j=0,...,\vartheta_{n}-1,
\end{align}
in which  $\delta_{r,j}$  is the Kronecker delta function.  Applying  (\ref{zcv}) in (\ref{Y2}) yields
\begin{align}\label{01899}
\psi_{r+\vartheta_{n}}(t)+\sum_{j=0}^{\vartheta_{n}-1}d_{r+\vartheta_{n},j}\varphi_{j}(t)&=\frac{1}{S_{r+\vartheta_{n},r}}\left(t^{r\alpha}-\sum_{\ell=r}^{r+\vartheta_{n}-1}S_{\ell,r}\left(\psi_{\ell}(t)+\sum_{j=0}^{\vartheta_{n}-1}d_{\ell,j}\varphi_{j}(t) \right)\right)\nonumber\\
&=\frac{1}{S_{r+\vartheta_{n},r}}\left(t^{r\alpha}-\sum_{\ell=r}^{r+\vartheta_{n}-1}S_{\ell,r}\psi_{\ell}(t)\right)\nonumber\\
&+\sum_{j=0}^{\vartheta_{n}-1}\left(\frac{-1}{S_{r+\vartheta_{n},r}}\left(\sum_{\ell=r}^{r+\vartheta_{n}-1}d_{\ell,j} S_{\ell,r}\right)\right)\varphi_{j}(t).
\end{align}
Comparison of the both sides of (\ref{01899}) gives
\begin{align}\label{LO}
\left\{
  \begin{array}{ll}
\psi_{r+\vartheta_{n}}(t)=\displaystyle\frac{1}{S_{r+\vartheta_{n},r}}\left(t^{r\alpha}-\displaystyle\sum_{\ell=r}^{r+\vartheta_{n}-1}S_{\ell,r}\psi_{\ell}(t) \right),\ \ r \geq 0,\\
  \displaystyle d_{r+\vartheta_{n},j}=\frac{-1}{S_{r+\vartheta_{n},r}}\left(\sum_{\ell=r}^{r+\vartheta_{n}-1}d_{\ell,j} S_{\ell,r}\right),\ \ \ r\geq 0,\ \
  j=0,...,\vartheta_{n}-1.
  \end{array}
\right.
\end{align}
The relations (\ref{Y1}) and (\ref{LO}) allow to generate the FC-polynomials by a simple recursive procedure.

\begin{thm} Assume that $\mathcal{H}_{n}(t)$ is of the from (\ref{bm1}) and all of the above notations and condition (\ref{Uniq}) hold. Then, the exact polynomial solution (Tau-solution) of Eq. (\ref{HOO}) is given by
 \begin{align}\label{mn}
y_{n}(t)=\sum_{r=0}^{n}g_{r}\psi_{r}(t)+\sum_{r=0}^{\vartheta_{n}-1}\tau_{n,r}\bigg(\sum_{j=0}^{n+\vartheta_{n}-r}C_{n+\vartheta_{n}-r,j}\psi_{j}(t)\bigg),
\end{align}
where $\tau_{n,r}$ are determined by solving the $\vartheta_{n}$-dimensional linear system of algebraic equations (called Tau-system)
\begin{equation}\label{Sys}
\mathcal{M}\overline{\tau}=\mathcal{B},
\end{equation}
where
\[
\mathcal{M}=\displaystyle\left(
  \begin{array}{cccc}
   \displaystyle \sum_{j=0}^{n+\vartheta_{n}}C_{n+\vartheta_{n},j}d_{j,0} &  \displaystyle \sum_{j=0}^{n+\vartheta_{n}-1}C_{n+\vartheta_{n}-1,j}d_{j,0} & \cdots &  \displaystyle \sum_{j=0}^{n+1}C_{n+1,j}d_{j,0} \\
     \displaystyle \sum_{j=0}^{n+\vartheta_{n}}C_{n+\vartheta_{n},j}d_{j,1} &  \displaystyle \sum_{j=0}^{n+\vartheta_{n}-1}C_{n+\vartheta_{n}-1,j}d_{j,1} & \cdots &  \displaystyle \sum_{j=0}^{n+1}C_{n+1,j}d_{j,1} \\
    \vdots & \vdots & \vdots & \vdots \\
 \displaystyle \sum_{j=0}^{n+\vartheta_{n}}C_{n+\vartheta_{n},j}d_{j,\vartheta_{n}-1} &  \displaystyle \sum_{j=0}^{n+\vartheta_{n}-1}C_{n+\vartheta_{n}-1,j}d_{j,\vartheta_{n}-1} & \cdots &  \displaystyle \sum_{j=0}^{n+1}C_{n+1,j}d_{j,\vartheta_{n}-1} \\
  \end{array}
\right),
\]
\[
\overline{\tau}=(\tau_{n,0}, \tau_{n,1},\cdots,\tau_{n,\vartheta_{n}-1})^{T},
\]
and
\[
\mathcal{B}=\big(- \displaystyle\sum_{r=0}^{n}g_{r}d_{r,0}
    ,-\displaystyle\sum_{r=0}^{n}g_{r}d_{r,1}
    ,\cdots,
    -\displaystyle\sum_{r=0}^{n}g_{r}d_{r,\vartheta_{n}-1}
\big)^{T}.
\]
\end{thm}

\begin{proof}
From relation (\ref{HOO}) and (\ref{bm1}), we have
 \begin{align}
\widetilde{\mathcal{L}}y_{n}(t)=\sum_{r=0}^{n}g_{r}t^{r\alpha}+\sum_{r\in \mathcal{S} }^{}\tau_{n,r}\bigg(\sum_{j=0}^{n+\vartheta_{n}-r}C_{n+\vartheta_{n}-r,j}t^{j\alpha}\bigg).
\end{align}
Thus by Definition of \ref{CBN} and linearity of $\widetilde{\mathcal{L}}$
 \begin{align}\label{bn1pp}
\widetilde{\mathcal{L}}\left(y_{n}(t)-\sum_{r\in \mathcal{S}}^{}\tau_{n,r}\bigg(\sum_{j=0}^{n+\vartheta_{n}-r}C_{n+\vartheta_{n}-r,j}\varphi_{j}(t)\bigg)-\sum_{r=0}^{n}g_{r}\varphi_{r}(t)\right)=0.
\end{align}
Because of Theorem \ref{ml}, i.e., $1 \notin \bigtriangleup_{\widetilde{\mathcal{K}}}$, we find
  \begin{equation}\label{opn}
y_{n}(t)=\sum_{r=0}^{n}g_{r}\varphi_{r}(t)+\sum_{r\in \mathcal{S}}^{}\tau_{n,r}\bigg(\sum_{j=0}^{n+\vartheta_{n}-r}C_{n+\vartheta_{n}-r,j}\varphi_{j}(t)\bigg),
\end{equation}
In view of relation (\ref{zcv}), we can rewrite (\ref{opn}) in the form
 \begin{align}\label{0521}
y_{n}(t)&=\sum_{r=0}^{n}g_{r}\psi_{r}(t)+\sum_{r\in \mathcal{S}}^{}\tau_{n,r}\bigg(\sum_{j=0}^{n+\vartheta_{n}-r}C_{n+\vartheta-r,j}\psi_{j}(t)\bigg)\nonumber\\
&+\sum_{\ell=0}^{\vartheta_{n}-1}\left(\sum_{r=0}^{n}g_{r}d_{r,\ell}+\sum_{r\in \mathcal{S}}^{}\tau_{n,r}\bigg(\sum_{j=0}^{n+\vartheta_{n}-r}C_{n+\vartheta_{n}-r,j}d_{j,\ell}\bigg) \right)\varphi_{\ell}(t).
\end{align}
In order to leave out the undefined canonical polynomials in (\ref{0521}) the parameters $\tau_{n,r}$ are determined such that the coefficients of $\varphi_{\ell}(t),\
\ell=0,...,\vartheta_{n}-1$ be equal to zero
\begin{align}\label{BBVV}
\sum_{r\in \mathcal{S}}^{}\tau_{n,r}\bigg(\sum_{j=0}^{n+\vartheta_{n}-r}C_{n+\vartheta_{n}-r,j}d_{j,\ell}\bigg)=-\sum_{r=0}^{n}g_{r}d_{r,\ell},\ \ \ \ell=0,...,\vartheta_{n}-1,
\end{align}
 therefore $\mathcal{S}=\lbrace 0,1,...,\vartheta_{n}-1\rbrace$ and Tau-system (\ref{Sys}) is derived. Finally,  the Tau-solution of  the perturbed problem (\ref{HOO}) becomes
\begin{align}
y_{n}(t)&=\sum_{r=0}^{n}g_{r}\psi_{r}(t)+\sum_{r=0}^{\vartheta_{n}-1}\tau_{n,r}\bigg(\sum_{j=0}^{n+\vartheta_{n}-r}C_{n+\vartheta-r,j}\psi_{j}(t)\bigg).
\end{align}
\end{proof}
\begin{cor}
From (\ref{Y1}) and (\ref{LO}), we have that
\begin{align*}
\deg[ \psi_{\vartheta_{n}}(t)]=0,\  \deg[\psi_{\vartheta_{n}+1}(t)]=\alpha,...,\deg[\psi_{r+\vartheta_{n}}(t)]=r\alpha,
\end{align*}
therefore,
\begin{align*}
\deg[y_{n}(t)]=\deg\bigg[\sum_{i=0}^{n}g_{i}\psi_{i}(t)+\sum_{i\in \mathcal{S}}^{}\tau_{n,i}\bigg(\sum_{j=0}^{n+\vartheta_{n}-i}C_{n+\vartheta_{n}-i,j}\psi_{j}(t)\bigg)\bigg]\leq n\alpha.
\end{align*}
\end{cor}
\begin{cor}For $\vartheta_{n}=0$, the Tau-solution is obtained as follows
 \[
 y_{n}(t)=\sum_{r=0}^{n}g_{r}\varphi_{r}(t).
 \]
 \end{cor}

\begin{cor}
For $g\equiv 0$, from relation (\ref{mn}) we have
\[
 y_{n}\equiv 0 \ \ \Leftrightarrow  \ \ \ \overline{\tau}={\mathbf{0}}.
 \]
\end{cor}
\begin{cor}
The dimension of the Tau-system remains fixed and independent of
 the degree of the Tau-solution if $H(t,s)\in \widehat{\mathbb{W}}^{\alpha}_{p}$ for $p\in \mathds{N}_{0}$.
 \end{cor}

\section{Convergence  analysis }\label{Sec4}
In this section, we provide a convergence analysis of the method.
\begin{thm}
Let $y_{n}$ and $y$ be the solution of Eqs. (\ref{HOO}) and (\ref{100}),  and $\widetilde{g}$ and $\widetilde{H}$ be the best approximation of the functions $g$ and $H$, respectively. Then,
\[
\Vert y-y_{n}\Vert _{\infty}\rightarrow 0,\ \ \ n\rightarrow \infty.
\]
\end{thm}
\begin{proof}
Subtracting (\ref{HOO}) from (\ref{100}) yields
\begin{equation}\label{ppv}
y-y_{n}=g-\widetilde{g}+\mathcal{K}y-\widetilde{\mathcal{K}}y_{n}-\mathcal{H}_{n}.
\end{equation}
By setting $ e_{n}=y-y_{n}$ we obtain
 \begin{align}\label{0p00}
\mathcal{K}y-\widetilde{\mathcal{K}}y_{n}&=\mathcal{K}y-\widetilde{\mathcal{K}}y+\widetilde{\mathcal{K}}(y-y_{n})=\mathcal{K}y-\widetilde{\mathcal{K}}y+\widetilde{\mathcal{K}}e_{n}
\end{align}
From (\ref{ppv}) and (\ref{0p00}) it follows that
\begin{align}
e_{n}=\int_{0}^{t}t^{-\beta}(t-s)^{\gamma -1}s^{\beta-\gamma}H(t,s)e_{n}(s)ds+J_{1}+J_{2}-J_{3}-\mathcal{H}_{n},
\end{align}
where
\[
J_{1}=g-\widetilde{g},\ \ J_{2}=(\mathcal{K}-\widetilde{\mathcal{K}})y,\ \
J_{3}=(\mathcal{K}-\widetilde{\mathcal{K}})e_{n}.
\]
Consequently,
\begin{equation}
\vert e_{n}(t)\vert \leq \Vert H \Vert_{\infty} \int_{0}^{t}t^{-\beta}(t-s)^{\gamma -1}s^{\beta-\gamma}\vert e_{n}(s)\vert ds+\vert J_{1}\vert +\vert J_{2}\vert+\vert J_{3}\vert+\vert\mathcal{H}_{n}\vert.
\end{equation}
By the Gronwall's inequality (Lemma 3.5; Ref. \cite{Xiaohua}), we obtain
\begin{equation}
\Vert e_{n}\Vert_{\infty} \leq C\bigg(  \Vert J_{1}\Vert_{\infty} +\Vert J_{2}\Vert_{\infty}+\Vert J_{3}\Vert_{\infty}+\Vert\mathcal{H}_{n}\Vert_{\infty}\bigg).
\end{equation}
By Theorem (\ref{098}), we obtain
\begin{equation}\label{bvvb}
\Vert J_{1}\Vert_{\infty}\leq \frac{M}{\Gamma((n+1)\theta+1)}\rightarrow 0,\ \ \ n\rightarrow \infty,
\end{equation}
where $M=\max \vert \mathbb{\partial}_{t}^{(n+1)\theta}g(t) \vert$. Thanks to Theorem (\ref{uuzz}), we have
\begin{align}\label{6y}
\vert \left(\mathcal{K}-\widetilde{\mathcal{K}}\right)u\vert  & \leq  \int_{0}^{t}t^{-\beta}(t-s)^{\gamma -1}s^{\beta-\gamma}\vert H(t,s)-\widetilde{H}(t,s)\vert \vert u(s)\vert ds\nonumber \\
&\leq \Vert H-\widetilde{H}\Vert_{\infty} \Vert u\Vert_\infty  B(\gamma,\beta-\gamma+1)\rightarrow 0, \ \ n\rightarrow \infty,
\end{align}
so,
\begin{equation}\label{mnmm}
\Vert J_{2}\Vert_{\infty}, \ \Vert J_{3}\Vert_{\infty}\rightarrow 0,\ \ \ n \rightarrow \infty.
\end{equation}
According to \cite{camjournal}, the sequence of dual spaces $\lbrace {\mathbb{W}^{\alpha}_{n}}^{\perp} \rbrace$ have the following property
\[
 \cdots {\mathbb{W}^{\alpha}_{n+2}}^{\perp}\subset {\mathbb{W}^{\alpha}_{n+1}}^{\perp}\subset {\mathbb{W}^{\alpha}_{n}}^{\perp};\ \ \ diam({\mathbb{W}^{\alpha}_{n}}^{\perp})\rightarrow 0,\ \ n\rightarrow \infty,
\]
therefore,
\begin{equation}\label{jo}
\Vert \mathcal{H}_{n}\Vert_{\infty} \rightarrow 0,\ \ \ n \rightarrow \infty,
\end{equation}
because of $\mathcal{H}_{n}(t) \in {\mathbb{W}^{\alpha}_{n}}^{\perp}$. Finally, from relations (\ref{bvvb})-(\ref{jo}) the desired result is obtained.
\end{proof}

\section{Numerical results}\label{Sec5}
This section contain some examples to illustrate the significance of the method. These examples were also studied in recent works \cite{Cai,Nemati,Allaei1,Shayanfard,Wang,Nemati2}. All of the numerical calculation are performed on computer using a program written in Maple 2018. Here, we state an algorithm to summarize the steps of the method:
\begin{algorithm}
\renewcommand{\thealgorithm}{}
\floatname{algorithm}{}
\caption{{\bf{Algorithm:}} Fractional recursive Tau method}
{\bf{Input:}} Function $f(t)$, $H(t,s)$ and the values of $\beta$ and $\gamma$.\\
{\bf{Step 1:}} Compute the values of  $\sigma_{1}$, $\sigma_{2}$ and $\alpha$.\\
{\bf{Step 2:}} Compute $\widetilde{\mathcal{L}}(t^{r\alpha})=\displaystyle\sum_{\ell=r}^{r+\vartheta_{n}}S_{\ell,r}t^{\ell\alpha}$.\\
{\bf{Step 3:}}  Construct $ \psi_{r}(t)$ and $ \displaystyle d_{r,j}$ for $r\geq 0$ and $j\in \mathcal{S}$.\\
{\bf{Step 4:}} Solve the linear system (\ref{Sys}).\\
{\bf{Output:}} Construct the Tau-solution $y_{n}(t)$.
\end{algorithm}
\begin{ex}\cite{Usta,Shayanfard} Consider the following third kind Volterra integral equation
\begin{equation}\label{p0}
t^{\frac{1}{2}}y(t)=t^{\frac{4}{2}}-B(\frac{1}{2},\frac{9}{2})t^{\frac{8}{2}}+\int_{0}^{t}(t-s)^{\frac{-1}{2}}s^{2}y(s)ds,
\end{equation}
with the exact solution $y(t)=t^{\frac{3}{2}}$.  Based on the Theorem \ref{ml}, the associated cordial Volterra integral operator is
\[
\mathcal{K}y(t)=\int_{0}^{t}t^{-\frac{1}{2}}(t-s)^{\frac{-1}{2}}s^{2}y(s)ds
\]
is compact and the Eq. (\ref{p0}) has unique solution.  Now, we apply the algorithm to obtain the Tau-solution. Thus
\begin{align}
\left\{
  \begin{array}{ll}
  \psi_{r+4}(t)=\displaystyle\frac{1}{k_{r+\vartheta,r}}\left(t^{r\alpha}-\displaystyle\sum_{\ell=r}^{r+\vartheta-1}k_{\ell,r}\psi_{\ell}(t) \right),\ \ r \geq 0,\\
  \displaystyle d_{r+4,j}=\frac{-1}{k_{r+\vartheta,r}}\left(\sum_{\ell=r}^{r+\vartheta-1}d_{\ell,j} k_{\ell,r}\right),\ \ \ r\geq 0,\ \
  j=0,...,3,
  \end{array}
\right.
\end{align}
where
\[
\psi_{0}(t)=\cdots=\psi_{3}(t)=0,\ \ \displaystyle d_{i,j}=\delta_{i,j},\ \ i,j=0,...,3,
\]
and so for $n=8$, we obtain $\tau_{8,0}=\tau_{8,1}=\tau_{8,2}=\tau_{8,3}=0$. Then,
\begin{align*}
y_{8}(t)&=\sum_{r=0}^{8}g_{r}\psi_{r}(t)+\sum_{r=0}^{3}\tau_{8,r}\bigg(\sum_{j=0}^{12-r}C_{12-r,j}\psi_{j}(t)\bigg),\\
&=t^{\frac{3}{2}}
\end{align*}
 which is the exact solution of Eq. (\ref{p0}). The numerical results of Refs. \cite{Usta,Shayanfard} show that the our method give the exact solution in comparison with numerical methods.
\end{ex}
\begin{figure}
 \centering
  \includegraphics[width=0.70\textwidth]{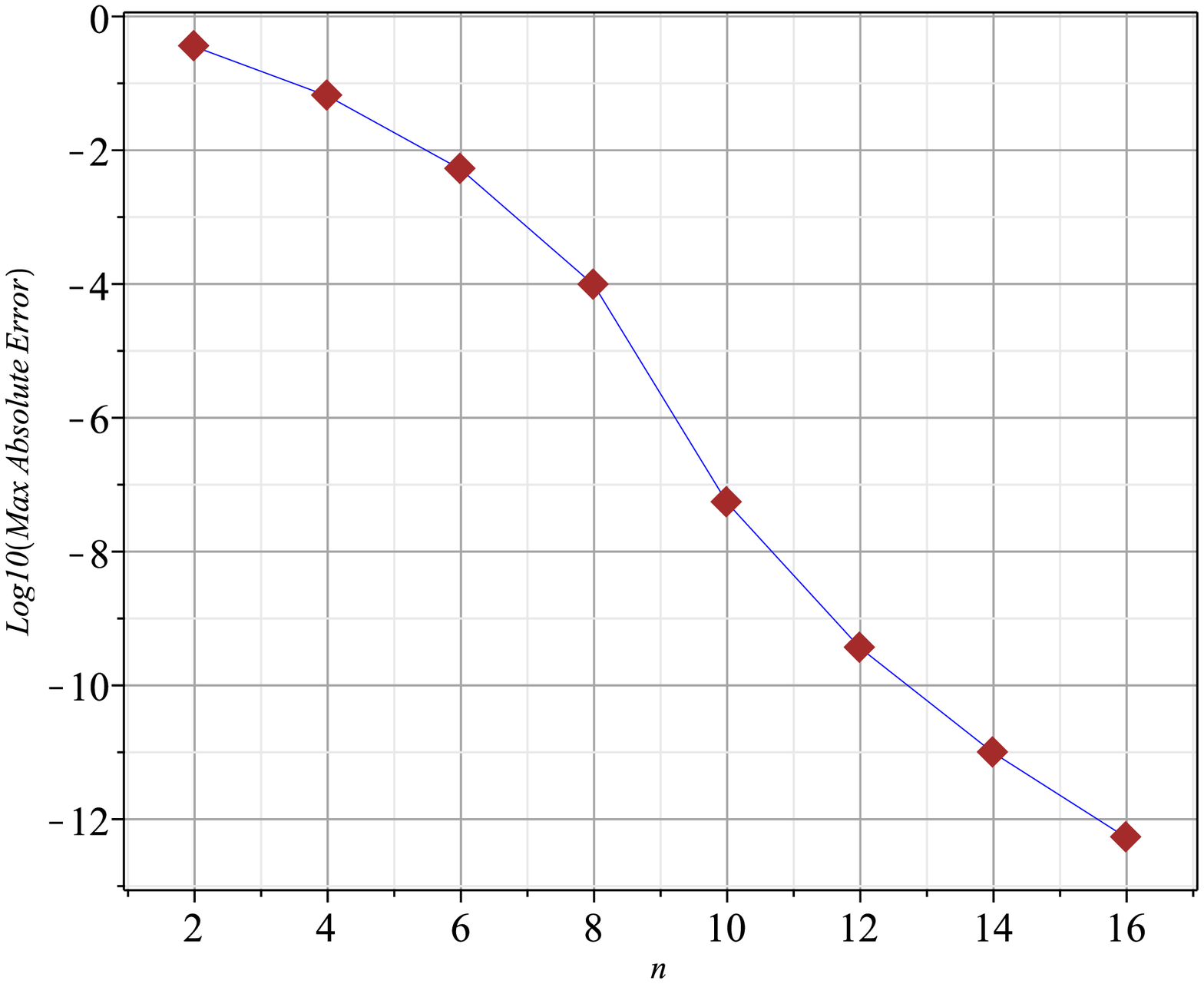}
\caption{The plot of $\Vert e \Vert_{n} $  for various $n$ for Example 2.}
\label{Af1}
\end{figure}
\begin{figure}
 \centering
  \includegraphics[width=1\textwidth]{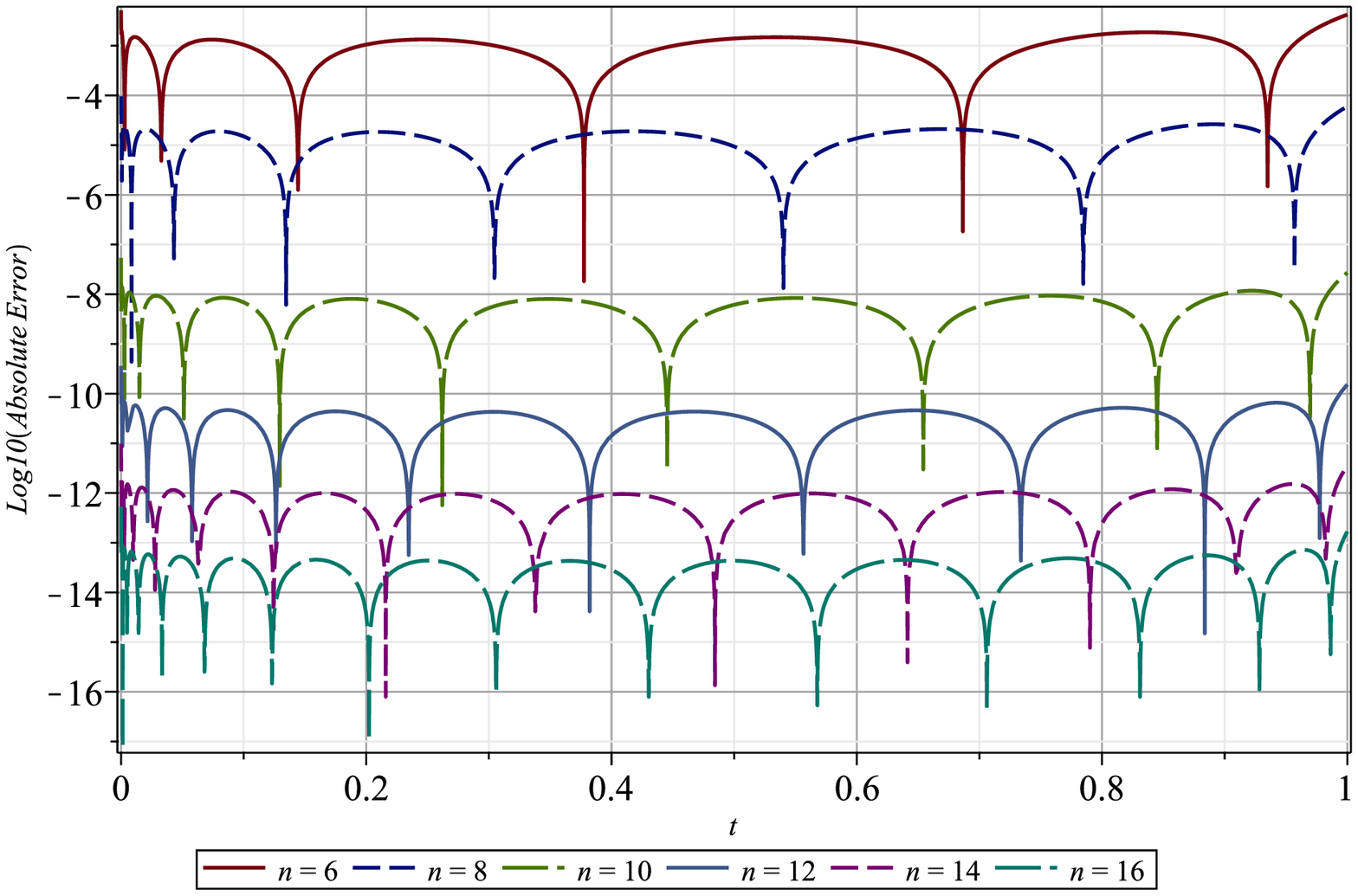}
\caption{Absolute error function $\vert e_{n}(t)\vert $  for various $n$ for Example 2.}
\label{Af1E}
\end{figure}
\begin{ex} Consider the following third kind Volterra integral equation
\begin{equation}\label{vn}
t^{\frac{2}{3}}y(t)=t^{\frac{47}{12}}\left(1-\frac{\Gamma(\frac{1}{3})\Gamma(\frac{55}{12})}{\pi\sqrt{3}\Gamma(\frac{59}{12})}\right)+\frac{\sqrt{3}}{3\pi}\int_{0}^{t}(t-s)^{\frac{-2}{3}}s^{\frac{1}{3}}y(s)ds,\ \ \
\end{equation}
and the exact solution $y(t)=t^{\frac{13}{4}}$.  The integral equation (\ref{vn}) is a well-known Lighthill model that describes the temperature distribution on the surface of a projectile moving through a laminar layer \cite{Lighthill}. Based on the Theorem \ref{ml}, the associated cordial Volterra integral operator
\[
\mathcal{K}y(t)=\int_{0}^{t}t^{-\frac{2}{3}}(t-s)^{\frac{-2}{3}}\frac{\sqrt{3}s^{\frac{1}{3}}}{3\pi}y(s)ds
\]
is non-compact with the uncountable spectrum
\[
\bigtriangleup_{\mathcal{K}}=\lbrace 0\rbrace \cup \lbrace \frac{\sqrt{3}}{3\pi}B(\frac{1}{3},1+\frac{1}{3}+\lambda); \ Re (\lambda)\geq 0 \rbrace,
\]
therefore, the Eq. (\ref{vn}) has unique solution if $1 \notin \bigtriangleup$, i.e.,
\[
1-\frac{\sqrt{3}}{3\pi}B(\frac{1}{3},1+\frac{1}{3}+\lambda)\neq 0.
\]
By implementation of the algorithm, we obtain
\[
\left\{
  \begin{array}{ll}
Q_{r}(t)=\frac{1}{1-\frac{\sqrt{3}}{3\pi}B\left(\frac{1}{3}, \frac{1+r}{3} +1\right)}t^{\frac{r}{3}},\ \ \ r\geq 0,\\
y_{n}(t)=\sum_{r=0}^{n}g_{r}Q_{r}(t).
  \end{array}
\right.
\]
The linear variation of error versus the degree of the Tau-solution in
semi-log representation is displayed in Fig. \ref{Af1}.
The behavior of absolute error function for different values of $n$ on the interval $[0,1]$ is shown in Fig. \ref{Af1E}.
Tab. \ref{TT11} shows the comparison of maximum absolute error of the method with the  Ref. \cite{Cai}. Also, the results of Tab. \ref{TT12} show the high accuracy of method in comparison with the other methods.

\begin{table}[!h]
\caption{Comparison the values of $\Vert e_{n}\Vert_{\infty} $ of our method and Ref.
\cite{Cai}  versus $n$ for Example 2}
\begin{center}
 \begin{tabular}{c||cccccc}
 \hline\noalign{\smallskip}
n &6 & 8&10&12&14&16   \\\hline\hline
Our method &5.17e-03&9.61e-05&5.37e-08&3.61e-10&9.79e-12&5.28e-13\\
Ref. \cite{Cai} &1.09e-05&1.48e-06&3.15e-07&8.84e-08&2.99e-08&7.71e-08\\
\hline
\noalign{\smallskip}
\end{tabular}
\end{center}
\label{TT11}
\end{table}
\begin{table}[!h]
\caption{Comparison results of Example 2}
\begin{center}
 \begin{tabular}{c|cccc}
\hline\noalign{\smallskip}
 Our method& Ref. \cite{Allaei1} & Ref. \cite{Wang}&Ref\cite{Nemati} &Ref\cite{Nemati2}\\
 (n=14)&(m=3,\ N=256)&(m=3,N=256)&($\nu=\gamma$=0,\ M=5,\ k=6)&(n=192)\\\hline
9.79e-12&5.13e-9&3.66e-12&2.02-10&5.16e-9\\
\hline
\noalign{\smallskip}
\end{tabular}
\end{center}
\label{TT12}
\end{table}
\end{ex}
\begin{figure}
 \centering
  \includegraphics[width=0.80\textwidth]{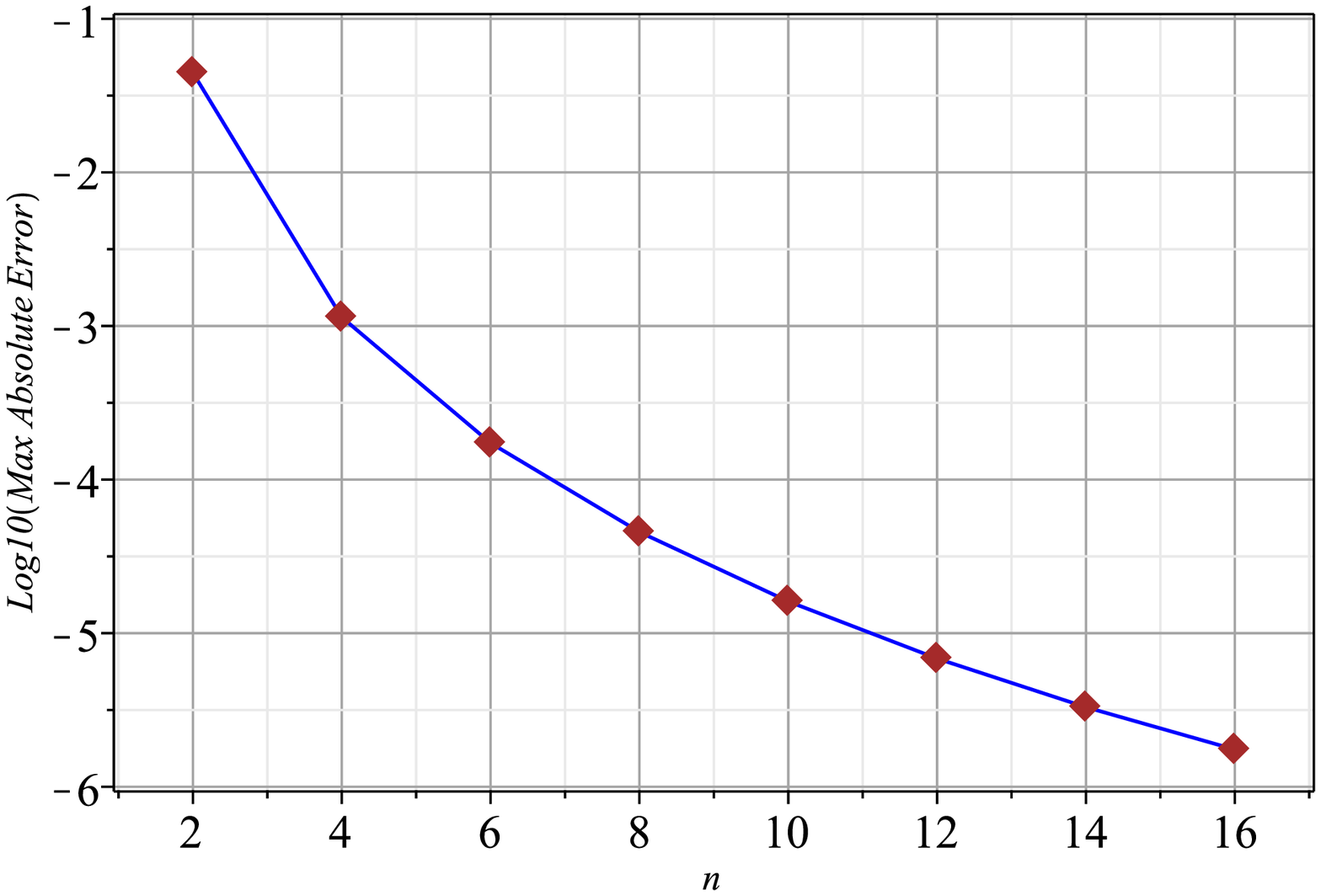}
\caption{The plot of $\Vert e \Vert_{n} $ for various $n$ for Example 3.}
\label{Af2}
\end{figure}
\begin{figure}
 \centering
  \includegraphics[width=1\textwidth]{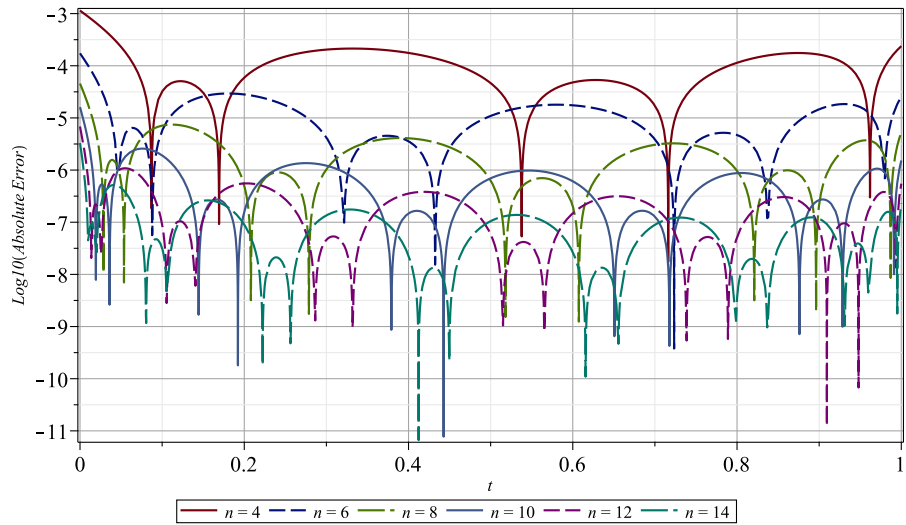}
\caption{Absolute error function $\vert e_{n}(t)\vert $  for various $n$ Example 3.}
\label{Af2E}
\end{figure}
\begin{ex}\cite{Allaei1}  Consider the following third kind Volterra integral equation
\begin{equation}\label{vn2}
ty(t)=\frac{6}{7}t^3\sqrt{t}+\int_{0}^{t}\frac{1}{2}y(s)ds,\ \ \
\end{equation}
with the exact solution $ y(t)=t^{\frac{5}{2}}$. This integral equations arise in the modeling of heat conduction problems with mixed-type boundary conditions problem. Based on the Theorem \ref{ml}, the associated cordial Volterra integral operator
\[
\mathcal{K}y(t)=\frac{1}{2}\int_{0}^{t}t^{-\frac{1}{2}}y(s)ds
\]
is non-compact with the uncountable spectrum
\[
\bigtriangleup_{\mathcal{K}}=\lbrace 0\rbrace \cup \lbrace \frac{1}{2(1+\lambda)}; \ Re (\lambda)\geq0 \rbrace,
\]
therefore, the Eq. (\ref{vn2}) has unique solution if $1 \notin \bigtriangleup$, i.e.,
\[
1-\frac{\sqrt{3}}{3\pi}B(\frac{1}{3},1+\frac{1}{3}+\lambda)\neq 0.
\]
The error of the Tau-solution for different values of $n$ are listed in Fig \ref{Af2} shows the spectral accuracy of our method for non-smooth solutions.
The behavior of absolute error function for different values of $n$ on the interval $[0,1]$ is shown in Fig. \ref{Af2E}.
Tab. \ref{TT21} and \ref{TT22} reports the efficiency of our method.
\begin{table}[!h]
\caption{The values of $\Vert e_{n}\Vert_{\infty} $ versus $n$ for Example 3}
\begin{center}
 \begin{tabular}{c|ccccccc}
 \hline\noalign{\smallskip}
n &4&6 & 8&10&12&14\\\hline
$\Vert e_{n}\Vert_{\infty} $ &1.14e-03&1.73e-04&4.56e-05&1.61e-05&6.84e-06&3.30e-6\\
\hline
\noalign{\smallskip}
\end{tabular}
\end{center}
\label{TT21}
\end{table}
\begin{table}[!h]
\caption{Comparison results of Example 3}
\begin{center}
 \begin{tabular}{c|cccc}
\hline\noalign{\smallskip}
 Our method& Ref. \cite{Allaei1} & Ref. \cite{Wang}&Ref\cite{Nemati} &Ref\cite{Nemati2}\\
 (n=20)&(m=2,\ N=256)&(m=2,N=256)&($\nu=\gamma$=0,\ M=5,\ k=6)&(n=192)\\\hline
6.02e-07&1.30e-5&2.46e-9&2.69-8&3.46e-8\\
\hline
\noalign{\smallskip}
\end{tabular}
\end{center}
\label{TT22}
\end{table}

\end{ex}
\begin{ex}\cite{Nemati}  Consider the following third kind Volterra integral equation
\begin{equation}\label{vn3}
t^{\frac{3}{2}}y(t)=t^{\frac{33}{10}}\left(1-\frac{\Gamma(\frac{19}{5})}{\sqrt{2\pi}\Gamma(\frac{43}{10})}\right)+\int_{0}^{t}\frac{\sqrt{2}}{2\pi}(t-s)^{\frac{-1}{2}}sy(s)ds;\ \ \
\end{equation}
with the exact solution $y(t)=t^{\frac{9}{5}}$. Based on the Theorem \ref{ml}, the associated cordial Volterra integral operator
\[
\mathcal{K}y(t)=\frac{\sqrt{2}}{2\pi}\int_{0}^{t}t^{-\frac{3}{2}}(t-s)^{\frac{-1}{2}}sy(s)ds
\]
is non-compact with the uncountable spectrum
\[
\bigtriangleup_{\mathcal{K}}=\lbrace 0\rbrace \cup \lbrace \frac{\sqrt{2}}{2\pi}B(\frac{1}{2},2+\lambda); \ Re (\lambda)\geq 0 \rbrace,
\]
therefore, the Eq. (\ref{vn3}) has unique solution if $1 \notin \bigtriangleup$, i.e.,
\[
1-\frac{\sqrt{2}}{2\pi}B(\frac{1}{2},2+\lambda)\neq 0.
\]
The numerical results of Fig \ref{Af3} and Tab. \ref{TT41} show the exponential rate of convergence of the method.
The behavior of absolute error function for different values of $n$ on the interval $[0,1]$ is shown in Fig. \ref{Af3E}.
In Tab. \ref{TT42},  the error norm of the method is compared with the error norm in the case of Jacobi wavelets method. The comparison results points up the accuracy of the method with a small number of basis FC-polynomials.
\begin{table}[!h]
\caption{The values of $\Vert e_{n}\Vert_{\infty} $ versus $n$ for Example 4}
\begin{center}
 \begin{tabular}{c|cccccccc}
 \hline\noalign{\smallskip}
n &6 & 8&10&12&14&16\\\hline
$\Vert e_{n}\Vert_{\infty}  $ &2.34e-05&3.08e-06&6.55e-07&1.86e-07&6.40e-8&2.54e-8\\
\hline
\noalign{\smallskip}
\end{tabular}
\end{center}
\label{TT41}
\end{table}
\begin{table}[!h]
\caption{Comparison results of Example 4}
\begin{center}
 \begin{tabular}{c|cccc}
\hline\noalign{\smallskip}
 Our method&& Ref. \cite{Nemati} with (M=5,\ k=5) \\
 (n=14)&($\nu=\gamma$=0.5 )&($\nu=\gamma$=0)&($\nu=\gamma$=-0.5)\\\hline
6.40e-8&4.99e-8&1.48e-7&4.99-7\\
\hline
\noalign{\smallskip}
\end{tabular}
\end{center}
\label{TT42}
\end{table}
\begin{figure}
 \centering
  \includegraphics[width=0.80\textwidth]{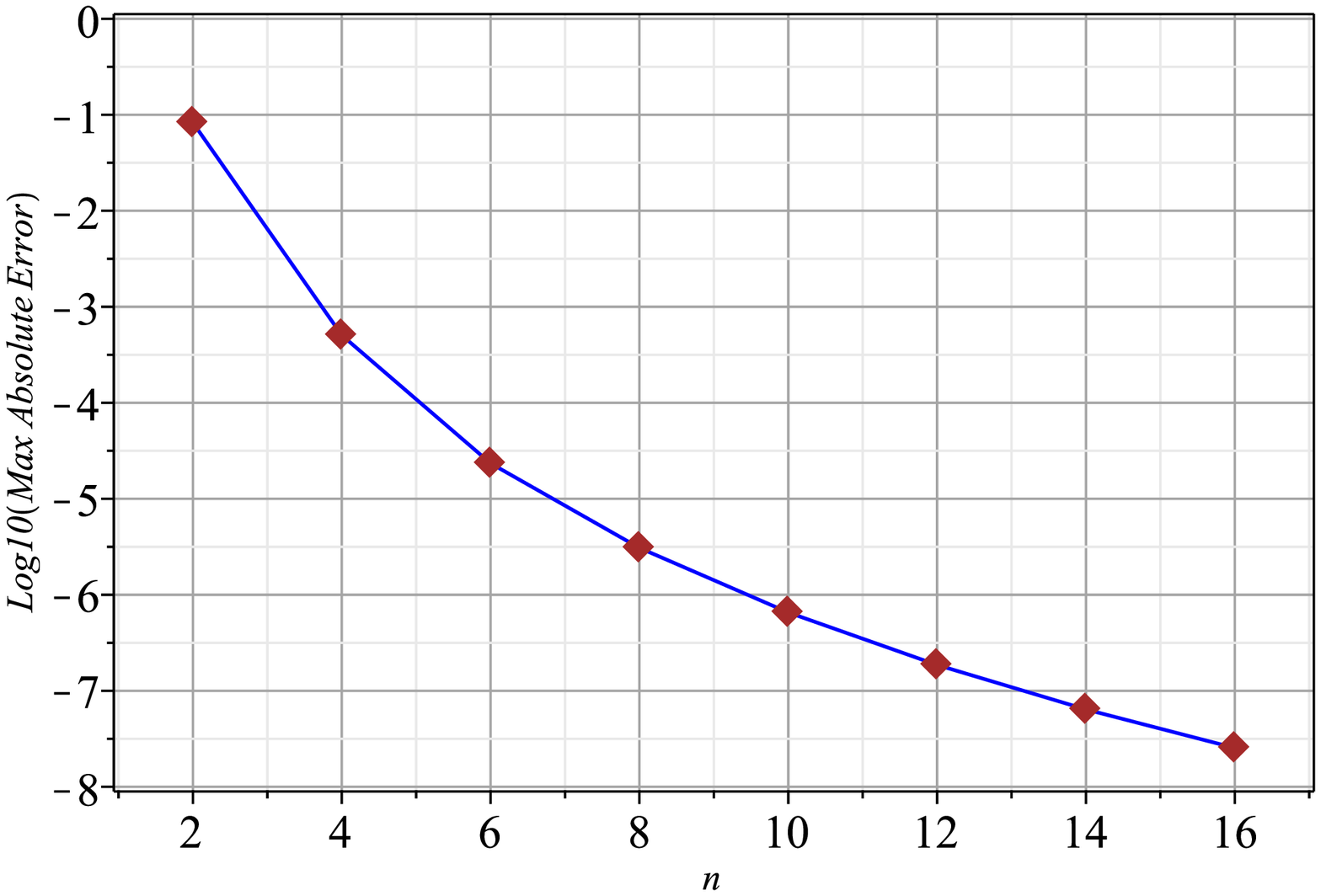}
\caption{The plot of $\Vert e \Vert_{n} $ for various $n$ for Example 4.}
\label{Af3}
\end{figure}
\begin{figure}
 \centering
  \includegraphics[width=1\textwidth]{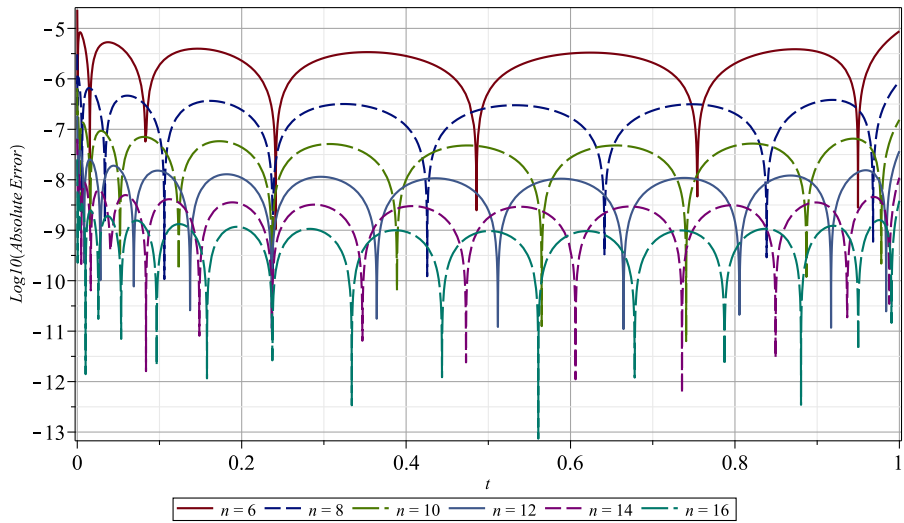}
\caption{Absolute error function $\vert e_{n}(t)\vert $  for various $n$ for Example 4.}
\label{Af3E}
\end{figure}
\end{ex}

\begin{ex}  Consider the following third kind Volterra integral equation
\begin{equation}\label{vn5}
ty(t)=g(t)+\int_{0}^{t}(t-s)^{\frac{-1}{3}}s^{2}y(s)ds
\end{equation}
 with the exact solution $y(t)=t^{\frac{1}{2}}\sin(t)$. Based on the Theorem \ref{ml}, the associated cordial Volterra integral operator
\[
\mathcal{K}y(t)=\int_{0}^{t}t^{-1}(t-s)^{\frac{-1}{3}}s^{2}y(s)ds
\]
is  compact and the  Eq. (\ref{vn5}) has unique solution. Fig. \ref{Af4} and  Tab. \ref{TT51} show the  error norm for different values of $n$. The behavior of absolute error function for different values of $n$ on the interval $[0,1]$ is shown in Fig. \ref{Af4E}.
They show that this method is effective even for small values of $n$.

\begin{table}[!h]
\caption{The values of $\Vert e_{n}\Vert_{\infty} $ versus $n$ for Example 5}
\begin{center}
 \begin{tabular}{c|cccccccc}
 \hline\noalign{\smallskip}
n &6 & 8&10&12&14&16\\\hline
$\Vert e_{n}\Vert_{\infty}  $ &8.29e-04&1.61e-05&2.64e-06&6.48e-08&4.18e-9&1.81e-9\\
\hline
\noalign{\smallskip}
\end{tabular}
\end{center}
\label{TT51}
\end{table}

\begin{figure}[ht]
 \centering
  \includegraphics[width=0.80\textwidth]{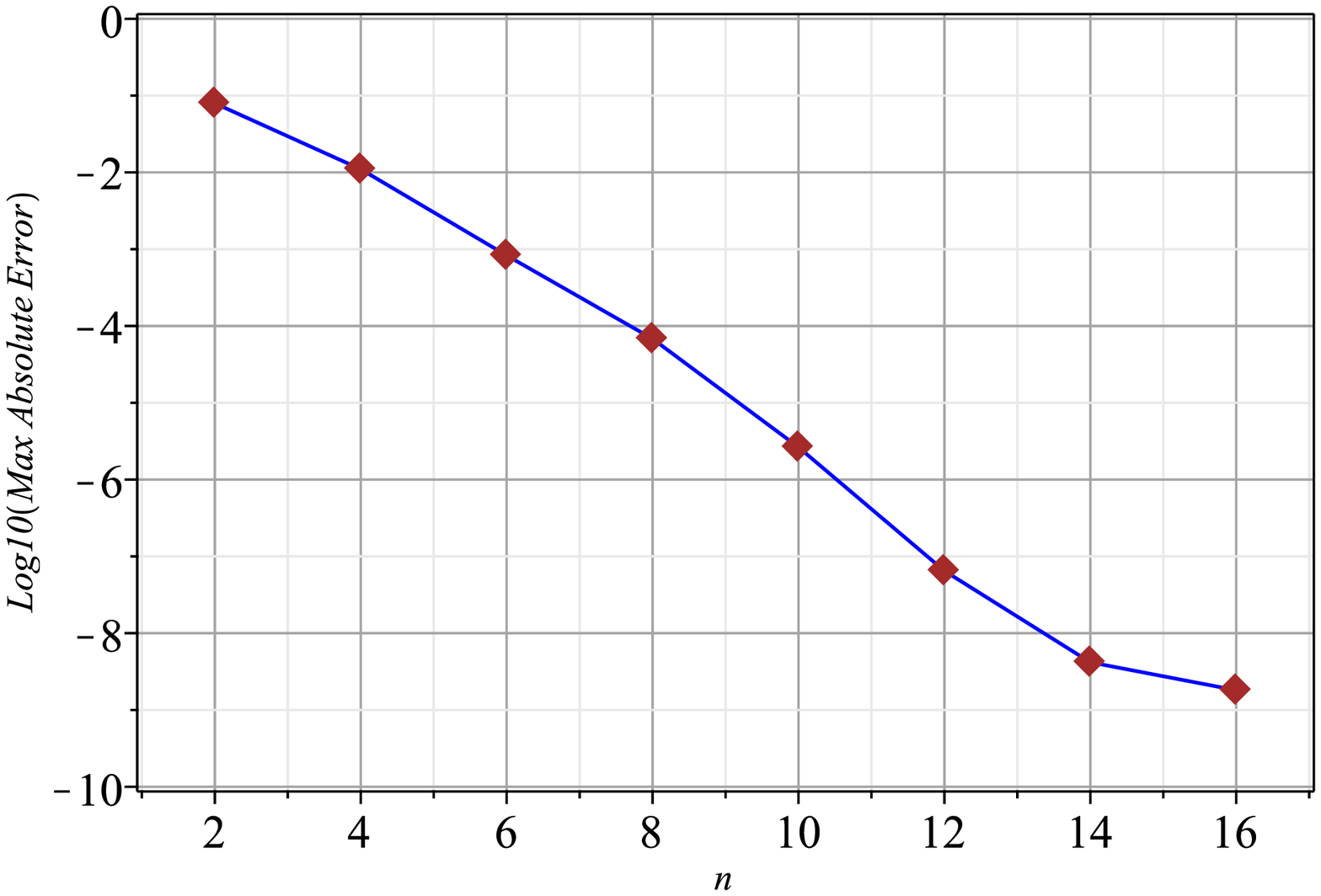}
\caption{The plot of $\Vert e \Vert_{n} $ for various $n$ for Example 5.}
\label{Af4}
\end{figure}
\begin{figure}
 \centering
  \includegraphics[width=1\textwidth]{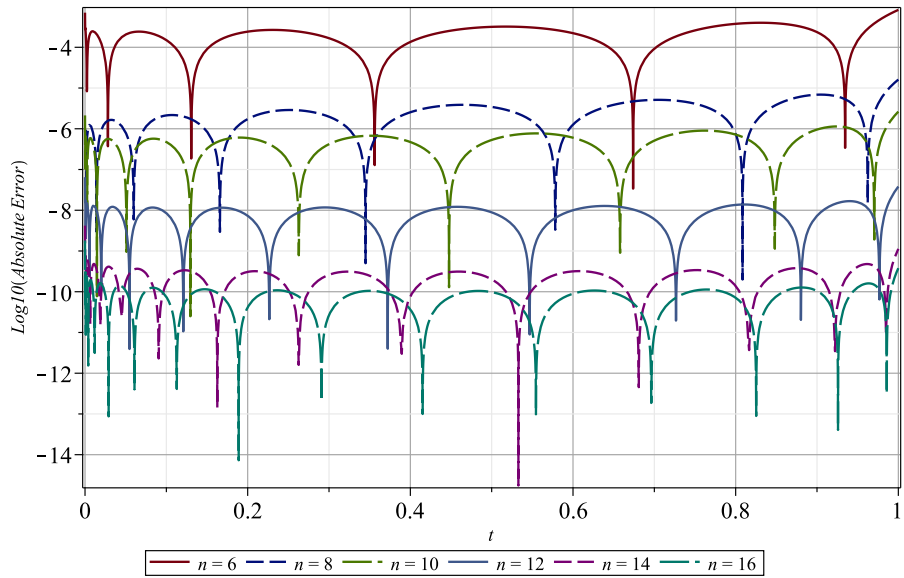}
\caption{Absolute error function $\vert e_{n}(t)\vert $  for various $n$ for Example 5.}
\label{Af4E}
\end{figure}
\end{ex}
\section{Conclusion and future works}\label{Sec6}
 The Tau recursive method is applied using a new class of fractional order polynomials.
 These fractional order polynomials are generated based on a simple recursive algorithm. The performance of this method, in comparison with existing techniques, is illustrated by a set of numerical examples. The success of this method results from the introduction of  fractional polynomials, which allow the Tau-solution to have a similar behavior to the one of the non-smooth solution. As a future work, this method can be extended to the linear/nonlinear integro-differential integral equations of the third kind.

\end{document}